\documentclass[12pt, reqno]{amsart}
\usepackage{amssymb,amsthm,amsmath}
\usepackage[numbers,sort&compress]{natbib}
\usepackage{amssymb,amsmath}
\usepackage{amsfonts}
\usepackage{mathrsfs}
\usepackage{latexsym}
\usepackage{amssymb}
\usepackage{amsthm}
\usepackage{indentfirst}
\let\oldsection\section
\renewcommand\section{\setcounter{equation}{0}\oldsection}

\newtheorem{corollary}{Corollary}[section]
\newtheorem{theorem}{Theorem}[section]
\newtheorem{lemma}{Lemma}[section]
\newtheorem{proposition}{Proposition}[section]
\newtheorem{definition}{Definition}[section]

\newtheorem{remark}{Remark}[section]

\date{January 6, 2014}

\begin{document}

\title[Strong Solutions to the 3D Primitive Equations]{Global Well-posedness of Strong Solutions\\ to the 3D Primitive Equations with Horizontal Eddy Diffusivity}

\author{Chongsheng~Cao}
\address[Chongsheng~Cao]{Department of Mathematics, Florida International University, University Park, Miami, FL 33199, USA}
\email{caoc@fiu.edu}

\author{Jinkai~Li}
\address[Jinkai~Li]{Department of Computer Science and Applied Mathematics, Weizmann Institute of Science, Rehovot 76100, Israel}
\email{jklimath@gmail.com}

\author{Edriss~S.~Titi}
\address[Edriss~S.~Titi]{ Department
of Computer Science and Applied Mathematics, Weizmann Institute of Science,
Rehovot 76100, Israel. Also Department of Mathematics and Department of Mechanical and Aerospace Engineering, University of California, Irvine, California 92697-3875, USA}
\email{etiti@math.uci.edu and edriss.titi@weizmann.ac.il}


\allowdisplaybreaks
\begin{abstract}
In this paper, we consider the initial-boundary value problem of the 3D primitive equations for oceanic and atmospheric dynamics with only horizontal diffusion in the temperature equation. Global well-posedness of strong solutions are established with $H^2$ initial data.
\end{abstract}

\maketitle

\allowdisplaybreaks
\section{Introduction}\label{sec1}

The primitive equations derived from the Boussinisq system of incompressible flow are fundamental models for weather prediction, see, e.g., Lewandowski \cite{LEWAN}, Majda \cite{MAJDA}, Pedlosky \cite{PED}, Vallis \cite{VALLIS}, and Washington and Parkinson \cite{WP}. Due to their importance, the primitive equations has been studied analytically by many authors, see, e.g., \cite{LTW92A,LTW92B,TZ04,PTZ09,MAJDA} and the references therein.

In this paper, we consider the viscous primitive equations with only horizontal
diffusion in the temperature equation:
\begin{eqnarray}
&\partial_tv+(v\cdot\nabla_H)v+w\partial_zv+\nabla_Hp-\Delta v+f_0k\times
v=0,\label{1.1-1}\\
&\partial_zp+T=0,\label{1.2-1}\\
&\nabla_H\cdot v+\partial_zw=0,\label{1.3-1}\\
&\partial_tT+(v\cdot\nabla_H)T+w\partial_zT-\Delta_HT=0,\label{1.4-1}
\end{eqnarray}
where the horizontal velocity $v=(v^1,v^2)$, the vertical velocity $w$, the temperature
$T$ and the pressure $p$ are the unknowns and $f_0$ is the Coriolis parameter. In this paper, we use the notations $\nabla_H=(\partial_x,\partial_y)$
and $\Delta_H=\partial_x^2+\partial_y^2$ to denote the horizontal gradient and the
horizontal Laplacian, respectively. The dominat horizontal eddy diffusivity in this model is justified by some geophysicists due to the strong horizontal turbulent mixing.

In 1990s, Lions, Temam and Wang \cite{LTW92A,LTW92B,LTW95} initialed the mathematical
studies on the primitive equations, where among other issues they established the global
existence of weak solutions. The uniqueness of weak solutions for 2D case was later proven
by Bresch, Guill\'en-Gonz\'alez, Masmoudi and Rodr\'iguez-Bellido \cite{BGMR03}; however,
the uniqueness of weak solutions in the three-dimensional case is still unclear. Local existence of strong solutions was obtained by Guill\'en-Gonz\'alez, Masmoudi and Rodr\'iguez-Bellido \cite{GMR01}. Global existence
of strong solutions for 2D case was established by Bresch, Kazhikhov and Lemoine in
\cite{BKL04} and Temam and Ziane in \cite{TZ04}, while the 3D case was established by
Cao and Titi \cite{CAOTITI2}. Global strong solutions for 3D case were also obtained by
Kobelkov \cite{KOB06} later by using a different approach, see also
the subsequent articles Kukavica and Ziane \cite{KZ07A,KZ07B}. The systems considered in all the
papers \cite{CAOTITI2,KOB06,KZ07A,KZ07B} are assumed to have diffusion in all
directions. It is proven by Cao and Titi \cite{CAOTITI3} that these global existence
results still hold true for the system with only vertical diffusion, provided the local in
time strong solutions exist. As the complement of \cite{CAOTITI3}, local existence
results for the system with only vertical diffusion are recently given by Cao, Li and Titi \cite{CAOLITITI} with $H^2$ initial data. Notably, the inviscid primitive equation, with or without coupling to the heat equation has been shown by Cao et al \cite{CINT} to blow up in finite time (see also \cite{TKW}).

In this paper, we consider the primitive equations with only horizontal diffusion in the temperature equation, i.e. system (\ref{1.1-1})--(\ref{1.4-1}).
The aim of this paper is to show that the strong solutions exist globally for system (\ref{1.1-1})--(\ref{1.4-1}) subject to some initial and boundary conditions.
More precisely, we consider the problem in the domain $\Omega_0=M\times(-h,0)$, with $M=(0,1)\times(0,1)$, and supplement system (\ref{1.1-1})--(\ref{1.4-1}) with the following boundary and initial conditions:
\begin{eqnarray}
& v, w\mbox{ and } T \mbox{ are }\mbox{periodic in }x \mbox{ and }y,\label{1.5-1}\\
&(\partial_zv,w)|_{z=-h,0}=(0,0),\quad T|_{z=-h}=1,\quad T|_{z=0}=0,\label{1.6-1}\\
&(v,T)|_{t=0}=(v_0, T_0). \label{1.7-1}
\end{eqnarray}

Replacing $T$ and $p$ by $T+\frac{z}{h}$ and $p-\frac{z^2}{2h}$, respectively, then system (\ref{1.1-1})--(\ref{1.4-1}) with (\ref{1.5-1})--(\ref{1.7-1}) is equivalent to the following system
\begin{eqnarray}
&\partial_tv+(v\cdot\nabla_H)v+w\partial_zv+\nabla_Hp-\Delta v+f_0k\times v=0,\label{1-1.1}\\
&\partial_zp+T=0,\label{1-1.2}\\
&\nabla_H\cdot v+\partial_zw=0,\label{1-1.3}\\
&\partial_tT+(v\cdot\nabla_H)T+w\left(\partial_zT+\frac{1}{h}\right)-\Delta_HT=0,\label{1-1.4}
\end{eqnarray}
subject to the boundary and initial conditions
\begin{eqnarray}
& v, w, T \mbox{ are }\mbox{periodic in }x \mbox{ and }y ,\label{1-1.5}\\
&(\partial_zv,w)|_{z=-h,0}=0,\quad T|_{z=-h,0}=0,\label{1-1.6}\\
&(v,T)|_{t=0}=(v_0, T_0). \label{1-1.7}
\end{eqnarray}
Here, for simplicity, we still use $T_0$ to denote the initial temperature in (\ref{1-1.7}), though it is now different from that in (\ref{1.7-1}).

Notice that the periodic subspace $\mathcal H$, given by
\begin{align*}
  \mathcal H:=&\{(v,w,p,T)|v,w,p\mbox{ and }T\mbox{are spatially periodic in all three variables} \\
  &\mbox{and even, odd, even and odd in }z \mbox{ variable},\mbox{ respectively}\},
\end{align*}
is invariant under the dynamics system (\ref{1-1.1})--(\ref{1-1.4}). That is if the initial data satisfy the properties stated in the definition of $\mathcal H$, then, as we will see later (see Theorem \ref{thm1}), the solutions to system (\ref{1-1.1})--(\ref{1-1.4}) will obey the same symmetry as the initial data. This motivated us to consider the following system
\begin{eqnarray}
&\partial_tv+(v\cdot\nabla_H)v+w\partial_zv+\nabla_Hp-\Delta v+f_0k\times v=0,\label{1.1}\\
&\partial_zp+T=0,\label{1.2}\\
&\nabla_H\cdot v+\partial_zw=0,\label{1.3}\\
&\partial_tT+(v\cdot\nabla_H)T+w\left(\partial_zT+\frac{1}{h}\right)-\Delta_HT=0,\label{1.4}
\end{eqnarray}
in $\Omega:=M\times(-h,h)$, subject to the boundary and initial conditions
\begin{eqnarray}
& v, w, p \mbox{ and } T \mbox{ are }\mbox{periodic in }x, y, z,\label{1.5}\\
& v\mbox{ and }p \mbox{ are even in }z,\mbox{ and } w\mbox{ and }T\mbox{ are odd in }z,\label{1.6}\\
&(v,T)|_{t=0}=(v_0, T_0). \label{1.7}
\end{eqnarray}
One can easily check that the restriction on the sub-domain $\Omega_0$ of a solution $(v,w,p,T)$ to system (\ref{1.1})--(\ref{1.7}) is a solution to the original system (\ref{1-1.1})--(\ref{1-1.7}).
Because of this, throughout this paper, we mainly concern on the study of system (\ref{1.1})--(\ref{1.7}) defined on $\Omega$, while the well-posedness results for system (\ref{1-1.1})--(\ref{1-1.7}) defined on $\Omega_0$ follow as a corollary of those for system (\ref{1.1})--(\ref{1.7}).

For any function $\phi(x,y,z)$ defined on $\Omega$, we define functions $\bar\phi$ and $\tilde\phi$ as follows
$$
\bar\phi(x,y)=\frac{1}{2h}\int_{-h}^h\phi(x,y,z)dz,\qquad\tilde\phi=\phi-\bar\phi.
$$
Using these notations, system (\ref{1.1})--(\ref{1.7})
is equivalent to (see \cite{CAOTITI3} for example)
\begin{eqnarray}
&\partial_tv-\Delta v+(v\cdot\nabla_H)v-\left(\int_{-h}^z\nabla_H\cdot v(x,y,\xi,t)d\xi\right)\partial_zv\nonumber\\
&\quad\qquad+f_0k\times v+\nabla_H\left(p_s(x,y,t)-\int_{-h}^zT(x,y,\xi,t)d\xi\right)=0,\label{1.8}\\
&\nabla_H\cdot\bar v=0,\label{1.9}\\
&\partial_tT-\Delta_HT+v\cdot\nabla_HT-\left(\int_{-h}^z\nabla_H\cdot v(x,y,\xi,t)d\xi\right)\left(\partial_zT+\frac{1}{h}\right)=0,\label{1.10}
\end{eqnarray}
complemented with the following boundary and initial conditions
\begin{eqnarray}
&v\mbox{ and } T \mbox{ are periodic in }x, y, z,\label{1.11}\\
&v\mbox{ and } T \mbox{ are even and odd in }z,\mbox{ respectively},\label{1.12}\\
&(v,T)|_{t=0}=(v_0, T_0). \label{1.13}
\end{eqnarray}

Throughout this paper, we use $L^q(\Omega), L^q(M)$ and $W^{m,q}(\Omega), W^{m,q}(M)$ to denote the standard Lebesgue and Sobolev spaces, respectively. For $q=2$, we use $H^m$ instead of $W^{m,2}$. We use $W_{\text{per}}^{m,q}(\Omega)$ and $H^m_{\text{per}}$ to denote the spaces of periodic functions in $W^{m,q}(\Omega)$ and $H^m(\Omega)$, respectively. For simplicity, we use the same notations $L^p$ and $H^m$ to denote the $N$ product spaces $(L^p)^N$ and $(H^m)^N$, respectively. We always use $\|u\|_p$ to denote the $L^p$ norm of $u$.

Definitions of strong solution, maximal existence time and global strong solution to system (\ref{1.8})--(\ref{1.13}) are given by the following three definitions, respectively.

\begin{definition}\label{def1.1}
Given a positive number $t_0$. Let $v_0\in H^2(\Omega)$ and $T_0\in H^2(\Omega)$ be two periodic functions, such that they are even and odd in $z$, respectively. A couple $(v,T)$ is called a strong solution to system
(\ref{1.8})-(\ref{1.13}) (or equivalently (\ref{1.1})--(\ref{1.7})) on $\Omega\times(0,t_0)$ if

(i) $v$ and $T$ are periodic in $x,y,z$, and they are even and odd in $z$, respectively;

(ii) $v$ and $T$ have the regularities
\begin{align*}
&v\in L^\infty(0,t_0; H^2(\Omega))\cap C([0,t_0];H^1(\Omega))\cap L^2(0,t_0; H^3(\Omega)), \\
&T\in L^\infty(0,t_0; H^2(\Omega))\cap C([0,t_0];H^1(\Omega)),\quad\nabla_HT\in L^2(0,t_0; H^2(\Omega)), \\
&\partial_tv\in L^2(0,t_0; H^1(\Omega)),\quad\partial_tT\in L^2(0,t_0; H^1(\Omega));
\end{align*}

(iii) $v$ and $T$ satisfies (\ref{1.8})--(\ref{1.10}) a.e. in $\Omega\times(0,t_0)$ and the initial condition (\ref{1.13}).
\end{definition}

\begin{definition}
A finite positive number $\mathcal{T}^*$ is called the maximal existence time of a strong solution $(v,T)$ to system (\ref{1.8})--(\ref{1.13}) if $(v,T)$ is a strong solution to the system on
$\Omega\times(0, t_0)$ for any $t_0<\mathcal T^*$ and $
\displaystyle\lim_{t\rightarrow \mathcal T_-^*}(\|v\|_{H^2}^2+\|T\|_{H^2}^2)=\infty.$
\end{definition}

\begin{definition}
A couple $(v,T)$ is called a global strong solution to system (\ref{1.8})--(\ref{1.13}) if it is a strong solution on $\Omega\times(0,t_0)$ for any $t_0<\infty$.
\end{definition}

The main result of this paper is the following global existence result.

\begin{theorem}\label{thm1}
Suppose that the periodic functions $v_0,T_0\in H^2(\Omega)$ are even and odd in $z$, respectively. Then system (\ref{1.8})-(\ref{1.13}) (or equivalently (\ref{1.1})--(\ref{1.7})) has a unique global strong solution $(v,T)$.
\end{theorem}

Local existence of strong solutions are obtained by a regularization mechanism.
More precisely, we add the vertical diffusion term, with a diffusion coefficient $\varepsilon>0$,
in the temperature equation to obtain
a regularized system. We then establish uniform estimates, in $\varepsilon$, for
strong solutions of the regularized system, over a short interval of time independent of $\varepsilon$, and then take the limit, as $\varepsilon$ goes to zero, to obtain
local strong solutions to system (\ref{1.8})--(\ref{1.13}). To obtain the global strong
solutions, from the local existence results, we need to establish the a priori estimates
on the derivatives of the solution, up to the second order. The first crucial estimate is the $L^6$
estimate on $v$, which has been originally obtained by Cao and Titi in
\cite{CAOTITI2,CAOTITI3}. Next, we establish the estimates on the derivatives.
Resulting from the lack of sufficient information on the equation for the vertical
velocity $w$, and the absence of the vertical
diffusion in the temperature equation, the treatments of different derivatives will vary. In particular, when we deal with the derivatives of
$v$ of the same order, we first work on the vertical derivatives and then the horizontal ones. The reason for this is due to the fact
that we need the estimates on the vertical derivatives to handle the term of the form
$\left(\int_{-h}^z\nabla_H\cdot v d\xi\right)\partial_zv$, which has "stronger nonlinearity"
than the term of the form $(v\cdot\nabla_H)v$. Keeping this in mind and making use of the
$L^6$ estimates for $v$, we successfully obtain the estimates on $\partial_z^2v$, then on
$\nabla_H\partial_zv$ and finally on $\nabla^2v$ and $\nabla^2T$. As a result, we obtain
the a priori estimates which guarantee the global existence of strong solutions.

As a corollary of Theorem \ref{thm1}, we have the following theorem, which states the well-posedness of strong solutions to system (\ref{1-1.1})--(\ref{1-1.7}). The strong solutions to system (\ref{1-1.1})--(\ref{1-1.7}) are defined in the similar way as before.

\begin{theorem}\label{thm2}
Let $v_0$ and $T_0$ be two functions such that they are periodic in $x$ and $y$. Denote by $v_0^{ext}$ and $T_0^{ext}$ the even and odd extensions in $z$ of $v_0$ and $T_0$, respectively. Suppose that $v_0^{ext}, T_0^{ext}\in H^2_{\text{per}}(\Omega)$. Then system (\ref{1-1.1})--(\ref{1-1.7}) has a unique global strong solution $(v,T)$.
\end{theorem}

The existence part follows directly by applying Theorem \ref{thm1} with initial data $(v_0^{\text{ext}}, T_0^{\text{ext}})$ and restricting the solution on the sub-domain $\Omega_0$. While the uniqueness part can be proven in the same way as that for Theorem \ref{thm1}. It should be pointed out that, due to the same reason as stated in Remark 1.1 in \cite{CAOLITITI}, the condition that $v_0^{ext}, T_0^{ext}\in H^2_{\text{per}}(\Omega)$ in the above theorem is necessary for the existence of strong solutions to system (\ref{1-1.1})--(\ref{1-1.7}).

The rest of the paper is organized as follows: in the next section, section \ref{sec2}, we prove the local existence of strong solutions; in section \ref{sec3}, by establishing the necessary a priori estimates, we show that the local strong solution can be extended to be a global one, and thus obtain a global strong solution.

Throughout this paper, we use $C$ to denote a general constant which may be different from line to line.

\section{Local Existence of Strong Solutions}\label{sec2}

In this section, we establish the local existence of strong solutions to system (\ref{1.1})--(\ref{1.7}), or equivalently system (\ref{1.8})--(\ref{1.13}).

We first cite the following proposition on the local existence of strong solutions to the system with full diffusion (see Proposition 2.1 of \cite{CAOLITITI}).

\begin{proposition}\label{lem2.1}
Let $v_0\in H^2(\Omega)$ and $T_0\in H^2(\Omega)$ be two periodic functions, such that they are even and odd in $z$, respectively. Then for any given $\varepsilon>0$, there is a $t_\varepsilon>0$, depending on $\varepsilon$, and a unique strong solutions $(v_\varepsilon, T_\varepsilon)$, with $(v_\varepsilon, T_\varepsilon)\in L^\infty(0,t_\varepsilon; H^2(\Omega))\cap C([0,t_\varepsilon];H^1(\Omega))\cap L^2(0,t_\varepsilon; H^3(\Omega))$ and $(\partial_t v_\varepsilon, \partial_t T_\varepsilon)\in L^2(0,t_\varepsilon;H^1(\Omega))$, to the following system
\begin{eqnarray}
&\partial_tv-\Delta v+(v\cdot\nabla_H)v-\left(\int_{-h}^z\nabla_H\cdot v(x,y,\xi,t)d\xi\right)\partial_zv\nonumber\\
&\quad\qquad+f_0k\times v+\nabla_H\left(p_s(x,y,t)-\int_{-h}^zT(x,y,\xi,t)d\xi\right)=0,\label{2.1}\\
&\nabla_H\cdot\bar v=0,\label{2.2}\\
&\partial_tT-\Delta_HT-\varepsilon\partial_z^2T+v\cdot\nabla_HT-\left(\int_{-h}^z\nabla_H\cdot v(x,y,\xi,t)d\xi\right)\left(\partial_zT+\frac{1}{h}\right)=0,\label{2.3}
\end{eqnarray}
subject to the boundary and initial conditions (\ref{1.11})--(\ref{1.13}).
\end{proposition}

The following lemma will be used to obtain a uniform lower bound of the existence time, independent of $\varepsilon$, and the uniform in $\varepsilon$ estimates on the local strong solution $(v_\varepsilon, T_\varepsilon)$
obtained in Proposition \ref{lem2.1}. It also plays an important role in proving the uniqueness of strong solutions.

\begin{lemma} (see \cite{CAOTITI1})\label{lem2.2}
The following inequalities hold true
\begin{align*}
&\int_M\left(\int_{-h}^hf(x,y,z)dz\right)\left(\int_{-h}^hg(x,y,z)h(x,y,z)dz\right)dxdy\\
\leq&C\|f\|_2^{1/2}\left(\|f\|_2^{1/2}+\|\nabla_Hf\|_{2}^{1/2}\right)\|g\|_2\|h\|_2^{1/2}\left(
\|h\|_2^{1/2}+\|\nabla_Hh\|_2^{1/2}\right),
\end{align*}
and
\begin{align*}
&\int_M\left(\int_{-h}^hf(x,y,z)dz\right)\left(\int_{-h}^hg(x,y,z)h(x,y,z)dz\right)dxdy\\
\leq&C\|f\|_2\|g\|_2^{1/2}\left(\|g\|_2^{1/2}+\|\nabla_Hg\|_{2}^{1/2}\right)\|h\|_2^{1/2}\left(
\|h\|_2^{1/2}+\|\nabla_Hh\|_2^{1/2}\right),
\end{align*}
for every $f,g,h$ such that the right hand sides make sense and are finite.
\end{lemma}

We also need the following lemma on differentiation under the integral sign and integration by parts.

\begin{lemma}(see \cite{CAOLITITI})\label{lem2.0}
Let $f$ and $g$ be two spatial periodic functions such that
\begin{eqnarray*}
&f\in L^2(0,t_0; H^3(\Omega)),\quad\partial_tf\in L^2(0,t_0; H^1(\Omega)),\\
&g\in L^2(0,t_0; H^2(\Omega)),\quad\partial_tg\in L^2(0,t_0; L^2(\Omega)).
\end{eqnarray*}
Then it follows that
\begin{eqnarray*}
&&\frac{d}{dt}\int_\Omega|\Delta f|^2dxdydz=-2\int_\Omega\nabla\partial_tf\nabla\Delta fdxdydz,\\
&&\int_\Omega\nabla\partial_{x^i}^2f\nabla\Delta fdxdydz=\int_\Omega|\partial_{x^i}\Delta f|^2dxdydz
\end{eqnarray*}
and
\begin{eqnarray*}
&&\frac{d}{dt}\int_\Omega|\partial_{x^i} g|^2dxdydz=-2\int_\Omega\partial_tg\partial_{x^i}^2 gdxdydz,\\
&&\int_\Omega\partial_{x^i}^2g\partial_{x^j}^2 gdxdydz=\int_\Omega|\partial_{x^i}\partial_{x^j}g|^2dxdydz,
\end{eqnarray*}
for a.e. $t\in(0,t_0)$, where $x^i, x^j\in\{x,y,z\}$.
\end{lemma}

The next lemma is a version of the Aubin-Lions lemma.

\begin{lemma}\label{AL}
(Aubin-Lions Lemma, See Simon \cite{Simon} Corollary 4) Assume that $X, B$ and $Y$ are three Banach spaces, with $X\hookrightarrow\hookrightarrow B\hookrightarrow Y.$ Then it holds that

(i) If $F$ is a bounded subset of $L^p(0, T; X)$ where $1\leq p<\infty$, and $\frac{\partial F}{\partial t}=\left\{\frac{\partial f}{\partial t}|f\in F\right\}$ is bounded in $L^1(0, T; Y)$, then $F$ is relatively compact in $L^p(0, T; B)$;

(ii) If $F$ is bounded in $L^\infty(0, T; X)$ and $\frac{\partial F}{\partial t}$ is bounded in $L^r(0, T; Y)$ where $r>1$, then $F$ is relatively compact in $C([0, T]; B)$.
\end{lemma}

Now we provide a lower bound, in dependent of $\varepsilon$, for the existence time and establish the uniform, in $\varepsilon$, estimates for the solution $(v_\varepsilon, T_\varepsilon)$ obtained in Proposition \ref{lem2.1}. We have the following:

\begin{proposition}\label{lem2.3}
The local strong solution $(v_\varepsilon,T_\varepsilon)$ given by Proposition \ref{lem2.1} can be established on the interval $(0,t_0)$, such that
$$
\sup_{0\leq t\leq t_0}(\|v_\varepsilon\|_{H^2}^2+\|T_\varepsilon\|_{H^2}^2)+\int_0^{t_0}
(\|\nabla v_\varepsilon\|_{H^2}^2
+\|\nabla_HT_\varepsilon\|_{H^2}^2+\varepsilon\|\partial_zT_\varepsilon\|_{H^2}^2)dt\leq C
$$
and
$$
\int_0^{t_0}(\|\partial_tv_\varepsilon\|_{H^1}^2+\|\partial_tT_\varepsilon\|_{H^1}^2)dt
\leq C,
$$
where $t_0$ and $C$ are two positive constants independent of $\varepsilon$.
\end{proposition}

\begin{proof}
Suppose $(0,t_\varepsilon^*)$ is the maximal interval of existence of the local strong solution $(v_\varepsilon, T_\varepsilon)$. We are going to show that $t_\varepsilon^*>t_0$,
for some positive number $t_0$ independent of $\varepsilon$.

We focus in our analysis on the interval $(0,t_\varepsilon^*)$. Multiplying (\ref{2.1}) by $v_\varepsilon$ and (\ref{2.3}) by $T_\varepsilon$, respectively, and summing the resulting equations up, then it follows from integrating by parts and using (\ref{2.2}) that
\begin{align*}
&\frac{1}{2}\frac{d}{dt}\int_\Omega(|v_\varepsilon|^2+|T_\varepsilon|^2)dxdydz+\int_\Omega(|
\nabla v_\varepsilon|^2+|\nabla_HT_\varepsilon|^2+\varepsilon|\partial_z T|^2)dxdydz\\
=&\int_\Omega\bigg[\nabla_H\left(\int_{-h}^zT_\varepsilon
d\xi\right)v_\varepsilon+\frac{1}{h}
\left(\int_{-h}^z\nabla_H\cdot v_\varepsilon d\xi\right)T_\varepsilon\bigg]dxdydz.
\end{align*}
Applying the operator $\nabla$ to equations (\ref{2.1}) and (\ref{2.3}), multiplying the
resulting equations by $-\nabla\Delta v_\varepsilon$ and $-\nabla\Delta T_\varepsilon$,
respectively, summing these equalities up and integrating over $\Omega$, it follows from
integrating by parts and Lemma \ref{lem2.0} that
\begin{align*}
&\frac{1}{2}\frac{d}{dt}\int_\Omega(|\Delta v_\varepsilon|^2
+|\Delta T_\varepsilon|^2)dxdydz+\int_\Omega(|\nabla\Delta v_\varepsilon|^2
+|\nabla_H\Delta T_\varepsilon|^2+\varepsilon|\partial_z\Delta T_\varepsilon|^2)dxdydz\\
=&\int_\Omega\bigg\{\nabla\bigg[(v_\varepsilon
\cdot\nabla_H)v_\varepsilon-\left(\int_{-h}^z\nabla_H\cdot v_\varepsilon d\xi\right)\partial_zv_\varepsilon-\nabla_H\left(\int_{-h}^zT_\varepsilon d\xi\right)\bigg]:\nabla\Delta v_\varepsilon\\
&-\Delta\bigg[v_\varepsilon\cdot\nabla_HT_\varepsilon-\left(\int_{-h}^z\nabla_H\cdot v_\varepsilon d\xi\right)\left(\partial_zT_\varepsilon+\frac{1}{h}\right)\bigg]\Delta T_\varepsilon\bigg\}dxdydz.
\end{align*}
Summing the above two equalities up, then it follows from Lemma \ref{lem2.2}, the H\"older, Young, Sobolev and Poincar\'e inequalities that
\begin{align*}
&\frac{1}{2}\frac{d}{dt}\int_\Omega(|v_\varepsilon|^2+|\Delta v_\varepsilon|^2+|T_\varepsilon|^2
+|\Delta T_\varepsilon|^2)dxdydz\\
&+\int_\Omega(|\nabla v_\varepsilon|^2+|\nabla\Delta v_\varepsilon|^2+|\nabla_HT_\varepsilon|^2+|\nabla_H\Delta T_\varepsilon|^2+\varepsilon|\partial_zT_\varepsilon|^2+\varepsilon|\partial_z\Delta T_\varepsilon|^2)dxdydz\\
=&\int_\Omega\bigg\{\nabla_H\left(\int_{-h}^zT_\varepsilon d\xi\right)v_\varepsilon+\frac{1}{h}
\left(\int_{-h}^z\nabla_H\cdot v_\varepsilon d\xi\right)T_\varepsilon\\
&+\nabla\bigg[(v_\varepsilon
\cdot\nabla_H)v_\varepsilon-\left(\int_{-h}^z\nabla_H\cdot v_\varepsilon d\xi\right)\partial_zv_\varepsilon-\nabla_H\left(\int_{-h}^zT_\varepsilon d\xi\right)\bigg]:\nabla\Delta v_\varepsilon\\
&-\Delta\bigg[v_\varepsilon\cdot\nabla_HT_\varepsilon-\left(\int_{-h}^z\nabla_H\cdot v_\varepsilon d\xi\right)\left(\partial_zT_\varepsilon+\frac{1}{h}\right)\bigg]\Delta T_\varepsilon\bigg\}dxdydz\\
=&\int_\Omega\bigg\{\left(\int_{-h}^z\nabla_HT_\varepsilon d\xi\right)v_\varepsilon
+\frac{1}{h}\left(\int_{-h}^z\nabla_H\cdot v_\varepsilon d\xi\right)T_\varepsilon+\bigg[\nabla v_\varepsilon\cdot\nabla_Hv_\varepsilon\\
&+v_\varepsilon\cdot\nabla_H\nabla v_\varepsilon-\nabla\left(\int_{-h}^z\nabla_H\cdot v_\varepsilon d\xi\right)\partial_zv_\varepsilon-\left(\int_{-h}^z\nabla_H\cdot v_\varepsilon d\xi\right)\nabla\partial_zv_\varepsilon\\
&-\nabla_H\nabla\left(\int_{-h}^zT_\varepsilon d\xi\right)\bigg]:\nabla\Delta v_\varepsilon
-\bigg[\Delta v_\varepsilon\cdot\nabla_HT_\varepsilon+2\nabla v_\varepsilon\cdot\nabla_H\nabla T_\varepsilon\\
&-\left(\int_{-h}^z\nabla_H\cdot\Delta v_\varepsilon d\xi\right)(\partial_zT_\varepsilon+\frac{1}{h})-2\left(\int_{-h}^z\nabla\nabla_H\cdot v_\varepsilon d\xi\right)\nabla\partial_zT_\varepsilon\bigg]\Delta T_\varepsilon\bigg\}dxdydz\\
\leq&C\int_\Omega\bigg\{\left(\int_{-h}^h|\nabla_HT_\varepsilon|d\xi\right)|v_\varepsilon|+\left(
\int_{-h}^h|\nabla v_\varepsilon|d\xi\right)|T_\varepsilon|+\bigg[|\nabla v_\varepsilon|^2+|v_\varepsilon||\nabla^2v_\varepsilon|\\
&+\left(\int_{-h}^h|\nabla^2v_\varepsilon|d\xi\right)|\partial_zv_\varepsilon|+\left(\int_{-h}^h
|\nabla v_\varepsilon|d\xi\right)|\nabla^2v_\varepsilon|+\left(\int_{-h}^h|\nabla^2T_\varepsilon|
d\xi\right)\bigg]|\nabla\Delta v_\varepsilon|\\
&+\bigg[|\Delta v_\varepsilon||\nabla_HT_\varepsilon|+|\nabla v_\varepsilon||\nabla_H\nabla T_\varepsilon|+\left(\int_{-h}^h|\nabla_H\Delta v_\varepsilon|d\xi\right)\bigg]|\Delta T_\varepsilon|\bigg\}dxdydz\\
&+C\int_\Omega\int_M\bigg[\left(\int_{-h}^h|\nabla\Delta v_\varepsilon|d\xi\right)\left(\int_{-h}^h|\partial_zT_\varepsilon||\Delta T_\varepsilon|d\xi\right)\\
&+\left(\int_{-h}^h|\nabla^2v_\varepsilon|d\xi\right)\left(\int_{-h}^h|\nabla^2T_\varepsilon|^2d\xi
\right)\bigg]dxdydz\\
\leq&C\Big[\|\nabla_HT_\varepsilon\|_2\|v_\varepsilon\|_2+\|\nabla v_\varepsilon\|_2\|T_\varepsilon\|_2+(\|\nabla v_\varepsilon\|_4^2+\|v_\varepsilon\|_\infty\|\nabla^2v_\varepsilon\|_2\\
&+\|\nabla^2v_\varepsilon\|_3\|\nabla v_\varepsilon\|_6+\|\nabla^2T_\varepsilon\|_2)\|\nabla\Delta v_\varepsilon\|_2+(\|\nabla^2v_\varepsilon\|_3\|\nabla_HT_\varepsilon\|_6\\
&+\|\nabla v_\varepsilon\|_3\|\nabla_H\nabla T_\varepsilon\|_6+\|\nabla\Delta v_\varepsilon\|_2)
\|\Delta T_\varepsilon\|_2\Big]+C\|\nabla\Delta v_\varepsilon\|_2\|\partial_zT_\varepsilon\|_2^{1/2}\\
&\times(\|\partial_zT_\varepsilon\|_2^{1/2}+\|\nabla_H\partial_zT_\varepsilon\|_2^{1/2})\|\Delta T_\varepsilon\|_2^{1/2}(\|\Delta T_\varepsilon\|_2^{1/2}+\|\nabla_H\Delta T_\varepsilon\|_2^{1/2})\\
&+C\|\nabla^2v_\varepsilon\|_2\|\nabla^2T_\varepsilon\|_2(\|\nabla^2T_\varepsilon
\|_2+\|\nabla_H\nabla^2T_\varepsilon\|_2)\\
\leq&C\Big[\|v_\varepsilon\|_{H^2}^2+\|T_\varepsilon\|_{H^2}^2+
(\|v_\varepsilon\|_{H^2}^2+\|\Delta v_\varepsilon\|_2^{1/2}\|\nabla\Delta v_\varepsilon\|_2^{1/2}\|v_\varepsilon\|_{H^2}+\|T_\varepsilon\|_{H^2}^2)\|\nabla\Delta v_\varepsilon\|_2\\
&+(\|\Delta v_\varepsilon\|_2^{1/2}\|\nabla\Delta v_\varepsilon\|_2^{1/2}\|T_\varepsilon\|_{H^2}
+\|v_\varepsilon\|_{H^2}\|\nabla_H\Delta T_\varepsilon\|_2+\|\nabla\Delta v_\varepsilon\|_2)\\
&\times\|\Delta T_\varepsilon\|_2\Big]+C\|\nabla\Delta v_\varepsilon\|_2(\|T_\varepsilon\|_{H^2}^2+\|T_\varepsilon\|_{H^2}^{3/2}\|\nabla_H\Delta T_\varepsilon\|_2^{1/2})\\
&+C\|v_\varepsilon\|_{H^2}(\|T_\varepsilon\|_{H^2}^2+\|T_\varepsilon\|_{H^2}\|\nabla_H\Delta T_\varepsilon\|_2)\\
\leq&\frac{1}{2}(\|\nabla\Delta v_\varepsilon\|_2^2+\|\nabla_H\Delta T_\varepsilon\|_2^2)+C(1+\|v_\varepsilon\|_{H^2}^6+\|T_\varepsilon\|_{H^2}^6)
\end{align*}
from which, we obtain, for any $t\in(0,t_\varepsilon^*)$,
\begin{align*}
&\sup_{0\leq s\leq t}(\|v_\varepsilon\|_{H^2}^2+\|T_\varepsilon\|_{H^2}^2)+\int_0^t
(\|\nabla v_\varepsilon\|_{H^2}^2+\|\nabla_HT_\varepsilon\|_{H^2}^2+\varepsilon\|\partial_z
T_\varepsilon\|_{H^2}^2)ds\\
\leq&CC_0+C\int_0^t(1+\|v_\varepsilon\|_{H^2}^2+\|T_\varepsilon\|_{H^2}^2)^3ds,
\end{align*}
where $C_0=\|v_0\|_{H^2}^2+\|T_0\|_{H^2}^2+1$.

Setting
$$
f(t)=\sup_{0\leq s\leq t}(\|v_\varepsilon\|_{H^2}^2+\|T_\varepsilon\|_{H^2}^2+1)+\int_0^t
(\|\nabla v_\varepsilon\|_{H^2}^2+\|\nabla_HT_\varepsilon\|_{H^2}^2+\varepsilon\|\partial_z
T_\varepsilon\|_{H^2}^2)ds
$$
for $t\in[0,t_\varepsilon^*)$. Then one has
$$
f(t)\leq CC_0+C\int_0^t(f(s))^3ds,\quad t\in[0,t_\varepsilon^*).
$$
Set $F(t)=\int_0^t(f(t)^3)ds+1$, then we have
$$
F'(t)=(f(t))^3\leq C_1(F(t))^3,\quad\forall t\in[0,t_\varepsilon^*),
$$
where $C_1$ is a positive constant depending only on $h$ and
$(v_0,T_0)$. This inequality implies
$$
F(t)\leq\frac{1}{\sqrt{1-2C_1t}},\quad\forall t\in[0,t_\varepsilon^*)\cap[0,\frac{1}{2C_1}),
$$
and thus
\begin{align*}
&\sup_{0\leq s\leq t}(\|v_\varepsilon\|_{H^2}^2+\|T_\varepsilon\|_{H^2}^2)+\int_0^t
(\|\nabla v_\varepsilon\|_{H^2}^2+\|\nabla_HT_\varepsilon\|_{H^2}^2+\varepsilon\|\partial_z
T_\varepsilon\|_{H^2}^2)ds\\
\leq&CC_0+CF(t)\leq CC_0+\frac{C}{\sqrt{1-2C_1t}}\leq C(C_0+\sqrt 2),
\end{align*}
for any $t\in[0,t_\varepsilon^*)\cap[0,\frac{1}{4C_1}].$ Recalling that $t_\varepsilon^*$ is the
maximal existence time, the above inequality implies that $t_\varepsilon^*>\frac{1}{4C_1}$. Thus
we can take $t_0=\frac{1}{4C_1}$.

Thanks to the estimates we have just proved, one can use the same argument as in the last paragraph of the proof of Proposition 3.1 in \cite{CAOLITITI} to obtain the estimates on $\partial_tv_\varepsilon$ and $\partial_t T_\varepsilon$, and thus we omit the details here. This completes the proof.
\end{proof}

Now we can prove the local well-posedness of strong solutions to system (\ref{1.1})--(\ref{1.7}), or equivalently system (\ref{1.8})--(\ref{1.13}).

\begin{proposition} \label{prop2.1}
Let $v_0\in H^2(\Omega)$ and $T_0\in H^2(\Omega)$ be two periodic functions, such that they are even and odd in $z$, respectively. Then system (\ref{1.8})--(\ref{1.13}) has a unique strong solution $(v,T)$ in $\Omega\times(0,t_0)$, where $t_0$ is the same positive time stated in Proposition \ref{lem2.3}. Moreover, the strong solution depends continuously on the initial data.
\end{proposition}

\begin{proof}
By Proposition \ref{lem2.1} and Proposition \ref{lem2.3}, for any given $\varepsilon>0$, system (\ref{2.1})--(\ref{2.3}), subject to the boundary and initial conditions (\ref{1.11})--(\ref{1.13}), has a unique strong solution $(v_\varepsilon, T_\varepsilon)$ in $\Omega\times(0,t_0)$ such that
$$
\sup_{0\leq t\leq t_0}(\|v_\varepsilon\|_{H^2}^2+\|T_\varepsilon\|_{H^2}^2)+\int_0^{t_0}(\|\nabla v_\varepsilon\|_{H^2}^2+\|\nabla_H
T_\varepsilon\|_{H^2}^2+\varepsilon\|\partial_zT_\varepsilon\|_{H^2}^2)dt\leq C
$$
and
$$
\int_0^{t_0}(\|\partial_tv_\varepsilon\|_{H^1}^2+\|\partial_tT_\varepsilon\|_{H^1}^2)dt
\leq C,
$$
where $C$ is a constant independent of $\varepsilon$. On account of these estimates and applying Lemma \ref{AL}, there is a subsequence, still denoted by $\{(v_\varepsilon,T_\varepsilon)\}$, and $(v,T)$, such that
\begin{eqnarray*}
  &(v_\varepsilon,T_\varepsilon)\rightarrow (v,T),\quad\mbox{in }C([0,t_0];H^1(\Omega)),\\
  &(\nabla v_\varepsilon,\nabla_HT_\varepsilon)\rightarrow(\nabla v,\nabla_HT),\quad\mbox{in }L^2(0,t_0;H^1(\Omega)),\\
  &(v_\varepsilon,T_\varepsilon){\overset{*}{\rightharpoonup}}(v,T),\quad\mbox{in }L^\infty(0,t_0;H^2(\Omega)),\\
  &(\nabla v_\varepsilon,\nabla_HT_\varepsilon)\rightharpoonup (\nabla v,\nabla_HT),\quad\mbox{in }L^2(0,t_0;H^2(\Omega)),\\
  &(\partial_tv_\varepsilon,\partial_tT_\varepsilon)\rightharpoonup(\partial_tv,\partial_tT)
,\quad\mbox{in }L^2(\Omega\times(0,t_0)),
\end{eqnarray*}
where $\rightharpoonup$ and ${\overset{*}{\rightharpoonup}}$ are the weak and weak-$*$ convergence, respectively. Thanks to these convergence, one can easily show that $(v,T)$ is a strong solution to system (\ref{1.8})--(\ref{1.13}), or equivalently to system (\ref{1.1})--(\ref{1.7}). The continuous dependence on the initial data, in particular the uniqueness, are straightforward consequence  of Proposition \ref{propcont} (see Corollary \ref{CORO}, below).
\end{proof}

For the continuous dependence on the initial data, the solutions are not required to have as high regularities as stated in Definition \ref{def1.1}. In fact, we have the following:

\begin{proposition} \label{propcont}  Let $(v_1,T_1)$ and $(v_2,T_2)$ be two spatially periodic functions, satisfying the following regularity properties
\begin{eqnarray*}
  &(v_i,T_i)\in L^\infty(0,t_0; H^1(\Omega))\cap C([0,t_0];L^2(\Omega)),\\
  &(\partial_tv_i,\partial_tT_i,\delta\partial_z^2v_i)\in L^2(\Omega\times(0,t_0)),\quad(\nabla_Hv_i,\nabla_HT_i)\in L^2(0,t_0;H^1(\Omega)),
\end{eqnarray*}
$i=1,2$, where $\delta\geq0$ is a given constant. Set
\begin{equation}\label{phi}
\begin{split}
  \phi(t)=&1+\|v_2(t)\|_2^4+\|\partial_zv_2(t)\|_2^4+\|v_2(t)\|_2^2\|\nabla_Hv_2(t)\|_2^2\\
  &+\|\partial_zv_2(t)\|_2^2
  \|\nabla_H\partial_zv_2(t)\|_2^2+\|T_2(t)\|_2^4+\|\partial_zT_2(t)\|_2^4\\
  &+\|T_2(t)\|_2^2\|\nabla_HT_2(t)\|_2^2+\|\partial_zT_2(t)\|_2^2
  \|\nabla_H\partial_zT_2(t)\|_2^2,
\end{split}
\end{equation}
for any $t\in(0,t_0)$.
Suppose that both $(v_1, T_1)$ and $(v_2,T_2)$ satisfy the following system
\begin{eqnarray}
&\partial_tv-\Delta_H v-\delta\partial_z^2v+(v\cdot\nabla_H)v-\left(\int_{-h}^z\nabla_H\cdot v(x,y,\xi,t)d\xi\right)\partial_zv\nonumber\\
&+f_0k\times v+\nabla_H\left(p_s(x,y,t)-\int_{-h}^zT(x,y,\xi,t)d\xi\right)=0,\label{EQ-1}\\
&\nabla_H\cdot\bar v=0,\label{EQ-2}\\
&\partial_tT-\Delta_HT+v\cdot\nabla_HT-\left(\int_{-h}^z\nabla_H\cdot v(x,y,\xi,t)d\xi\right)\left(\partial_zT+\frac{1}{h}\right)=0,\label{EQ-3}
\end{eqnarray}
in $\Omega\times(0,t_0)$.

Setting $(v,T)=(v_1-v_2,T_1-T_2)$, then it follows that
\begin{align*}
\sup_{0\leq s\leq t}&(\|v\|_2^2+\|T\|_2^2)+\int_0^t(\|\nabla v\|_2^2+\|\nabla_HT\|_2^2+
\delta\|\partial_zv\|_2^2)ds\\
\leq& Ce^{C\int_0^t\phi(s)ds}(\|(v_1)_0-(v_2)_0\|_2^2+\|(T_1)_0-(T_2)_0\|_2^2),
\end{align*}
for any $t\in(0,t_0)$, where $((v_i)_0,(T_i)_0)$, $i=1,2$, are the initial values of $(v_i,T_i)$.
\end{proposition}

\begin{proof}
One can easily check that $(v,T)$ satisfies
\begin{eqnarray}
&\partial_tv-\Delta_H v-\delta\partial_z^2v+(v_1\cdot\nabla_H)v+(v\cdot\nabla_H)v_2\nonumber\\
&-\left(\int_{-h}^z\nabla_H\cdot v_1d\xi\right)\partial_zv-\left(\int_{-h}^z\nabla_H\cdot vd\xi\right)\partial_zv_2+f_0k\times v\nonumber\\
&+\nabla_Hp_s(x,y,t)-\nabla_H\left(\int_{-h}^zT(x,y,\xi,t)d\xi\right)=0,\label{2.5}\\
&\nabla_H\cdot\bar v=0,\label{2.5-1}\\
&\partial_tT-\Delta_HT+v_1\cdot\nabla_HT+v\cdot\nabla_H T_2-\left(\int_{-h}^z\nabla_H\cdot v_1d\xi\right)\partial_zT\nonumber\\
&-\left(\int_{-h}^z\nabla_H\cdot vd\xi\right)\left(\partial_zT_2+\frac{1}{h}\right)=0.\label{2.6}
\end{eqnarray}

Multiplying (\ref{2.5}) by $v$ and integrating over $\Omega$, then it follows from
integrating by parts and (\ref{2.5-1}) that
\begin{align}
&\frac{1}{2}\frac{d}{dt}\int_\Omega|v|^2dxdydz+\int_\Omega(|\nabla_H v|^2+\delta|\partial_zv\|_2^2)dxdydz\nonumber\\
=&\int_\Omega\bigg\{\bigg[\left(\int_{-h}^z\nabla_H\cdot vd\xi\right)\partial_zv_2-(v\cdot\nabla_H)v_2\bigg]\cdot
v
-\left(\int_{-h}^zTd\xi\right)\nabla_H\cdot v\bigg\}dxdydz. \label{2.7}
\end{align}
By Lemma \ref{lem2.2}, and using Young's inequality, we have the following estimates
\begin{align*}
&\left|\int_\Omega\left(\int_{-h}^z\nabla_H\cdot vd\xi\right)\partial_zv_2\cdot vdxdydz\right|\nonumber\\
\leq&\int_M\left(\int_{-h}^h|\nabla_Hv|dz\right)\left(\int_{-h}^h
|\partial_zv_2||v|dz\right)dxdy\\
\leq&C\|\nabla_Hv\|_2\|\partial_zv_2\|_2^{1/2}(\|\partial_zv_2\|_2^{1/2}+\|\nabla_H\partial_zv_2
\|_2^{1/2})\|v\|_2^{1/2}(\|v\|_2^{1/2}+\|\nabla_Hv\|_2^{1/2})\\
\leq&\frac{1}{8}\|\nabla_Hv\|_2^2 +C(1+\|\partial_zv_2\|_2^4+\|\partial_zv_2\|_2^2\|\partial_z\nabla_Hv_2\|_2^2)\|v\|_2^2\\
\leq&\frac{1}{8}\|\nabla_Hv\|_2^2+C\phi(t)\|v\|_2^2,
\end{align*}
for any $t\in(0,t_0)$.
Noticing that $|v_2(z)|\leq\frac{1}{2h}\int_{-h}^h|v_2(z)|dz+\int_{-h}^h|\partial_zv_2|dz$, it follows from integrating by parts, applying Lemma \ref{lem2.2}, and using Young's inequality that
\begin{align}
&\left|\int_\Omega (v\cdot\nabla_H)v_2\cdot vdxdydz\leq\int_\Omega|\nabla_Hv||v||v_2|dxdydz\right|\nonumber\\
\leq&C\int_M\left(\int_{-h}^h(|v_2|+|\partial_zv_2|)dz\right)\left(\int_{-h}^h|\nabla_Hv||v|dz
\right)dxdy\nonumber\\
\leq&C\|\partial_zv_2\|_2^{1/2}(\|\partial_zv_2\|_2^{1/2}+\|\nabla_H\partial_zv_2\|_2^{1/2})
\|\nabla_Hv\|_2\|v\|_2^{1/2}(\|v\|_2^{1/2}+\|\nabla_Hv\|_2^{1/2})\nonumber\\
&+C\|v_2\|_2^{1/2}(\|v_2\|_2^{1/2}+\|\nabla_H v_2\|_2^{1/2})
\|\nabla_Hv\|_2\|v\|_2^{1/2}(\|v\|_2^{1/2}+\|\nabla_Hv\|_2^{1/2})\nonumber\\
\leq&\frac{1}{8}\|\nabla_Hv\|_2^2+C(1+\|v_2\|_2^4+\|v_2\|_2^2\|\nabla_Hv_2\|_2^2+
\|\partial_zv_2\|_2^4+\|\partial_zv_2\|_2^2\|\partial_z\nabla_Hv_2\|_2^2)\|v\|_2^2\nonumber\\
\leq&\frac{1}{8}\|\nabla_Hv\|_2^2+C\phi(t)\|v\|_2^2,\label{star}
\end{align}
for any $t\in(0,t_0)$.
The above two inequalities, substituted into (\ref{2.7}), imply
\begin{equation}
\frac{d}{dt}\|v(t)\|_2^2+\frac{3}{2}\|\nabla_H v(t)\|^2+2\delta\|\partial_zv(t)\|_2^2
\leq C\phi(t)(\|v(t)\|_2^2+\|T(t)\|_2^2),\label{2.8}
\end{equation}
for any $t\in(0,t_0)$.

Multiplying (\ref{2.6}) by $T$ and integrating by parts yield
\begin{align}
&\frac{1}{2}\frac{d}{dt}\int_\Omega|T|^2dxdydz+\int_\Omega|\nabla_HT|^2dxdydz
\nonumber\\
=&-\int_\Omega\left[v\cdot\nabla_HT_2-\left(\int_{-h}^z\nabla_H\cdot vd\xi\right)\left(\partial_zT_2+\frac{1}{h}\right)\right]Tdxdydz. \label{2.9}
\end{align}
Using the fact that $|T_2(z)|\leq\frac{1}{2h}\int_{-h}^h|T_2(z)|dz+
\int_{-h}^h|\partial_zT_2|dz,$ the same argument as that for (\ref{star}) gives
\begin{align*}
&\left|\int_\Omega v\cdot\nabla_HT_2Tdxdydz\right|=\left|\int_\Omega T_2(\nabla_H\cdot v\, T+v\cdot\nabla_HT)dxdydz\right|\\
\leq&\left|\int_\Omega|\nabla_Hv||T||T_2|dxdydz\right|+\left|\int_\Omega |\nabla_HT||v||T_2|dxdydz\right|\\
\leq&C\left|\int_M\left(\int_{-h}^h(|T_2|+|\partial_zT_2|)dz\right)
\left(\int_{-h}^h|\nabla_Hv||T|dz\right)dxdy\right|\\
&+C\left|\int_M\left(\int_{-h}^h(|T_2|+|\partial_zT_2|)dz\right)
\left(\int_{-h}^h|v||\nabla_HT|dz\right)dxdy\right|\\
\leq&\frac{1}{4}(\|\nabla_Hv\|_2^2+
\|\nabla_H T\|_2^2)+C(1+\|T_2\|_2^4+\|T_2\|_2^2\|\nabla_HT_2\|_2^2+
\|\partial_zT_2\|_2^4\\
&+\|\partial_zT_2\|_2^2\|\partial_z\nabla_HT_2\|_2^2)
(\|v\|_2^2+\|T\|_2^2)\nonumber\\
\leq&\frac{1}{4}(\|\nabla_Hv\|_2^2+
\|\nabla_H T\|_2^2)+C\phi(t)(\|v\|_2^2+\|T\|_2^2),
\end{align*}
for any $t\in(0,t_0)$. Applying Lemma \ref{lem2.1} again, it follows from the Young inequality that
\begin{align*}
&\left|\int_\Omega\left(\int_{-h}^z\nabla_H\cdot vd\xi\right)\left(\partial_zT_2+\frac{1}{h}\right)Tdxdydz\right|\\
\leq&C\int_M\left(\int_{-h}^h|\nabla_Hv|dz\right)\left(\int_{-h}^h(|\partial_zT_2|+1)|T|dz\right)dxdy\\
\leq&C\|\nabla_Hv\|_2\|\partial_zT_2\|_2^{1/2}(\|\partial_zT_2\|_2^{1/2}+\|\nabla_H\partial_zT_2
\|_2^{1/2})\|T\|_2^{1/2}\\
&\times(\|T\|_2^{1/2}+\|\nabla_HT\|_2^{1/2})+C\|\nabla_Hv\|_2\|T\|_2\\
\leq&\frac{1}{4}(\|\nabla_Hv\|_2^2+\|\nabla_HT\|_2^2)+C(1+\|\partial_zT_2\|_2^4+\|\partial_zT_2\|_2^2
\|\nabla_H\partial_zT_2\|_2^2)\|T\|_2^2\\
\leq&\frac{1}{4}(\|\nabla_Hv\|_2^2+\|\nabla_HT\|_2^2)+C\phi(t)\|T\|_2^2,
\end{align*}
for any $t\in(0,t_0)$. Substituting the above two estimates into (\ref{2.9}), one has
\begin{align*}
&\frac{d}{dt}\|T(t)\|_2^2+\|\nabla_HT(t)\|^2
\leq \|\nabla_Hv(t)\|_2^2+C\phi(t)(\|v(t)\|_2^2+\|T(t)\|_2^2),
\end{align*}
for any $t\in(0,t_0)$.

Summing the above inequality up with (\ref{2.8}) leads to
\begin{align*}
&\frac{d}{dt}(\|v(t)\|_2^2+\|T(t)\|_2^2)+\frac{1}{2}(\|\nabla_Hv(t)\|_2^2+\|\nabla_HT(t)\|_2^2+
\delta\|\partial_zv(t)\|_2^2)\\
\leq& C\phi(t)(\|v(t)\|_2^2+\|T(t)\|_2^2),
\end{align*}
for any $t\in(0,t_0)$, which, by Gronwall's inequality, implies the conclusion.
\end{proof}

As a result of Proposition \ref{propcont}, we have the following corollary, which guarantees the uniqueness and continuous dependence on initial data of strong solutions to system (\ref{1.8})--(\ref{1.13}), or equivalently to system (\ref{1.1})--(\ref{1.7}).

\begin{corollary}
  \label{CORO}
Strong solution to system (\ref{1.8})--(\ref{1.13}) is unique and  depends continuously on the initial data.
\end{corollary}

\begin{proof}
Let $(v_1,T_1)$ and $(v_2,T_2)$ be two strong solutions to system (\ref{1.8})--(\ref{1.13}) on $\Omega\times(0,t_0)$. Recalling the regularity properties of strong solution $(v_2,T_2)$, the function $\phi(t)$ defined by (\ref{phi}) is integrable on the time interval $(0,t_0)$. By virtue of this fact, one can apply Proposition \ref{propcont} to obtain the continuous dependence on initial data, and in particular the uniqueness. This completes the proof.
\end{proof}

\begin{remark}
(i) By Proposition \ref{propcont}, neither the vertical viscosity, i.e. one can take $\delta=0$, nor the vertical diffusion, which is zero in our case, are necessary to guarantee the uniqueness, and continuous dependence on initial data, of the solutions that enjoy the regularity properties stated in the proposition.

(ii) The question of global existence of strong solutions to system (\ref{EQ-1})--(\ref{EQ-3}) with initial data in $H^1$, especially the case that $\delta=0$, is a subject of future work. However, once the existence is established, the uniqueness, and the continuous dependence on initial data, of such solutions will follow from Proposition \ref{propcont}.
\end{remark}

\section{Global Existence of Strong Solutions}\label{sec3}

In this section, we show that the local strong solution established for short time in section \ref{sec2} can
be in fact extended to be a global one.

Let $(v,T)$ be the unique strong solution obtained in Proposition \ref{prop2.1}. Suppose that $(0,\mathcal T^*)$ is the maximal interval of existence. If $\mathcal T^*=\infty$ there is nothing to prove. Therefore, for the next analysis, we assume by contradiction that $\mathcal T^*<\infty$, and we will focus our analysis on $(0,\mathcal T^*)$. We have the following three propositions which will provide the needed a priori estimates on $(v,T)$.

\begin{proposition}\label{lem3.1}
There is a bounded continuously increasing function $K_1(t)$, on $[0,\mathcal T^*)$, such that
\begin{align*}
&\sup_{0\leq s\leq t}(\|v\|_2^2+\|T\|_\infty^2+\|v\|_6^2+\|\nabla_H\bar v\|_{L^2(M)}^2+\|\partial_zv\|_6^2)\\
&+\int_0^t(\|\nabla v\|_2^2+\|\nabla_HT\|_2^2+\|\Delta_H\bar v\|_{L^2(M)})ds\leq K_1(t),
\end{align*}
for any $t\in[0,\mathcal T^*)$.
\end{proposition}

\begin{proof}
The conclusion follows directly from inequalities (59), (69), (91) and (103) in \cite{CAOTITI3}, with slight modifications. Thus we omit the proof here.
\end{proof}

Set $u=\partial_zv$, then it satisfies
\begin{eqnarray}
&\partial_tu-\Delta u+(v\cdot\nabla_H)u-\left(\int_{-h}^z\nabla_H\cdot v(x,y,\xi,t)d\xi\right)\partial_zu\nonumber\\
&+(u\cdot\nabla_H)v-(\nabla_H\cdot v)u+f_0k\times u-\nabla_HT=0, \label{3.1}
\end{eqnarray}
on $(0,\mathcal T^*)$.

\begin{proposition}\label{lem3.2}
There is a bounded continuously increasing function $K_2(t)$, on $[0,\mathcal T^*)$, such that
\begin{align*}
&\sup_{0\leq s\leq t}\|\nabla u\|_2^2+\int_0^t\|\nabla^2u\|_2^2ds\leq K_2(t),
\end{align*}
for any $t\in[0,\mathcal T^*). $
\end{proposition}

\begin{proof}
By the boundary conditions (\ref{1.11}) and (\ref{1.12}), $u=\partial_zv$ is $2h$ periodic and odd in the vertical variable $z$, and thus $u(x,y,-h,t)=-u(x,y,h,t)=-u(x,y,-h,t)$, which implies $u|_{z=-h}=0$ and $\nabla_Hu|_{z=-h}=0$, for any $t\in(0,\mathcal T^*)$. Multiplying equation (\ref{3.1}) by $-\partial_z^2u$ and integration by parts, using
Lemma \ref{lem2.0} and the fact that $|\nabla_Hu(x,y,z,t)|\leq\int_{-h}^h|\nabla_H\partial_zu(x,y,\xi,t)|d\xi$,
then it follows from the H\"older, Sobolev and Poincar\'e inequalities that
\begin{align*}
&\frac{1}{2}\frac{d}{dt}\int_\Omega|\partial_zu|^2dxdydz+\int_\Omega|\partial_z\nabla u|^2dxdydz\\
=&\int_\Omega\bigg[(v\cdot\nabla_H)u-\left(\int_{-h}^z\nabla_H\cdot vd\xi\right)\partial_zu\bigg]\cdot\partial_z^2udxdydz\\
&+\int_\Omega\big[(u\cdot\nabla_H)v-(\nabla_H\cdot v)u+f_0k\times u-\nabla_HT\big]\cdot\partial_z^2udxdydz\\
=&-\int_\Omega\big\{[2(u\cdot\nabla_H)u-2(\nabla_H\cdot v)\partial_zu+(\partial_zu\cdot\nabla_H)v\\
&-(\nabla_H\cdot u)u]\cdot\partial_zu+\nabla_HT\cdot\partial_z^2u\big\}dxdydz\\
=&-\int_\Omega\big\{[2(u\cdot\nabla_H)u-(\nabla_H\cdot u)u]\cdot\partial_zu+2v\cdot\nabla_H(|\partial_zu|^2)\\
&-\nabla_H\cdot\partial_zuv\cdot\partial_zu-(\partial_zu\cdot\nabla_H)\partial_zu\cdot v+\nabla_HT\cdot\partial_z^2u\big\}dxdydz\\
\leq&C\int_\Omega\bigg[|u|\left(\int_{-h}^h|\nabla_H\partial_zu|d\xi\right)|\partial_zu|
+|v||\nabla_H\partial_zu||\partial_zu|+|\nabla_HT||\partial_z^2u|\bigg]dxdydz\\
\leq&C\big[(\|u\|_6+\|v\|_6)\|\nabla_H\partial_zu\|_2\|\partial_zu\|_3+\|\nabla_HT\|_2
\|\partial_z^2u\|_2\big]\\
\leq&C\big[(\|u\|_6+\|v\|_6)\|\nabla_H\partial_zu\|_2\|\partial_zu\|_2^{1/2}\|\nabla
\partial_zu\|_2^{1/2}+\|\nabla_HT\|_2\|\partial_z^2u\|_2\big]\\
\leq&\frac{1}{2}\|\nabla\partial_zu\|_2^2+C(\|u\|_6^4+\|v\|_6^4)\|\partial_zu\|_2^2
+C\|\nabla_HT\|_2^2.
\end{align*}
Thanks to Proposition \ref{lem3.1}, this inequality gives
\begin{align}
&\sup_{0\leq s\leq t}\|\partial_zu\|_2^2+\int_0^t\|\nabla\partial_zu\|_2^2ds\nonumber\\
\leq&e^{C\int_0^t(\|u\|_6^4+\|v\|_6^4)ds}\left(\|\partial_zu_0\|_2^2+C\int_0^t
\|\nabla_HT\|_2^2ds\right)\nonumber\\
\leq&Ce^{K_1^2(t)t}(\|v_0\|_{H^2}^2+K_1(t))=:K_2'(t), \label{3.2}
\end{align}
for every $t\in[0,\mathcal T^*)$.

Multiplying (\ref{3.1}) by $-\Delta_Hu$ and integrating by parts, using Lemma \ref{lem2.0} and the fact that
$|\nabla_Hv(x,y,z,t)|\leq|\nabla_H\bar v(x,y,t)|+\int_{-h}^h|\nabla_Hu(x,y,\xi,t)|d\xi$, then it follows from the H\"older, Sobolev, Poincar\'e inequality and Lemma \ref{lem2.2} that
\begin{align*}
&\frac{1}{2}\frac{d}{dt}\int_\Omega|\nabla_Hu|^2dxdydz+\int_\Omega|\nabla_H\nabla u|^2dxdydz\\
=&\int_\Omega\bigg[(v\cdot\nabla_H)u-\left(\int_{-h}^z\nabla_H\cdot vd\xi\right)\partial_zu+(u\cdot\nabla_H)v\\
&-(\nabla_H\cdot v)u+f_0k\times u-\nabla_HT\bigg]\cdot\Delta_Hudxdydz\\
\leq&C\int_\Omega\bigg[|v||\nabla_Hu|+\left(\int_{-h}^h|\nabla_Hu|dz+|\nabla_H\bar v|\right)(|\partial_zu|+|u|)\\
&+|\nabla_HT|\bigg]|\Delta_Hu|dxdydz\\
\leq&C\big[\|v\|_6\|\nabla_Hu\|_3+\|u\|_6(\|\nabla_Hu\|_3+\|\nabla_H\bar v\|_3)+\|\nabla_HT\|_2\big]\|\Delta_Hu\|_2\\
&+C\int_M\left[\left(\int_{-h}^h|\nabla_Hu|d\xi\right)+|\nabla_H\bar v|\right]\left(\int_{-h}^h
|\partial_zu||\Delta_Hu|d\xi\right)dxdy\\
\leq&C\big[(\|u\|_6+\|v\|_6)\|\nabla_Hu\|_3+\|u\|_6\|\nabla_H\bar v\|_3+\|\nabla_HT\|_2\big]\|\Delta_Hu\|_2\\
&+C\big[\|\nabla_Hu\|_2^{1/2}(\|\nabla_Hu\|_2^{1/2}+\|\nabla^2_Hu\|_2^{1/2})+
\|\nabla_H\bar v\|_2^{1/2}(\|\nabla_H\bar v\|_2^{1/2}+\|\nabla_H^2\bar v\|_2^{1/2})\big]\\
&\times\|\partial_z u\|_2^{1/2}(\|\partial_zu\|_2^{1/2}+\|\nabla_H\partial_zu\|_2^{1/2})\|\Delta_Hu\|_2\\
\leq&C\big[(\|u\|_6+\|v\|_6)\|\nabla_Hu\|_2^{1/2}\|\nabla\nabla_Hu\|_2^{1/2}+\|u\|_6\|\nabla_H\bar v\|_{L^2(M)}^{2/3}\|\Delta_H\bar v\|_{L^2(M)}^{1/3}\\
&+\|\nabla_HT\|_2\big]\|\Delta_Hu\|_2+C(\|\nabla_H\bar v\|_{L^2(M)}^{1/2}\|\Delta_H\bar v\|_{L^2(M)}^{1/2}+\|\nabla_Hu\|_2^{1/2}\|\nabla\nabla_Hu\|_2^{1/2})\\
&\times\|\partial_zu\|_2^{1/2}\|\nabla\partial_zu\|_2^{1/2}\|\Delta_Hu\|_2\\
\leq&\frac{1}{2}\|\nabla_H\nabla u\|_2^2+C(\|u\|_6^4+\|v\|_6^4)\|\nabla_Hu\|_2^2+C(\|u\|_6^2\|\Delta_H\bar v\|_{L^2(M)}^2+\|\nabla_HT\|_{L^2(M)}^2)\\
&+C\|\partial_zu\|_2^2\|\nabla\partial_zu\|_2^2\|\nabla_Hu\|_2^2+C\|\nabla_H\bar v\|_{L^2(M)}\|\Delta_H\bar v\|_{L^2(M)}\|\partial_zu\|_2\|\nabla\partial_zu\|_2\\
\leq&\frac{1}{2}\|\nabla_H\nabla u\|_2^2+C(\|u\|_6^4+\|v\|_6^4+\|\partial_zu\|_2^2\|\nabla\partial_zu\|_2^2)\|\nabla_Hu\|_2^2\\
&+C(\|u\|_6^2\|\Delta_H\bar v\|_{L^2(M)}^2+\|\nabla_HT\|_2^2+\|\nabla_H\bar v\|_{L^2(M)}^2\|\Delta_H\bar v\|_{L^2(M)}^2+\|\partial_zu\|_2^2\|\nabla\partial_zu\|_2^2).
\end{align*}
Therefore, by Proposition \ref{lem3.1}, and using (\ref{3.2}), it follows from the above inequality that for any $t\in[0,\mathcal \mathcal T^*)$
\begin{align*}
&\sup_{0\leq s\leq t}\|\nabla_Hu\|_2^2+\int_0^t\|\nabla\nabla_Hu\|_2^2ds\\
\leq&e^{C\int_0^t(\|u\|_6^4+\|v\|_6^4+\|\partial_zu\|_2^2\|\nabla\partial_zu\|_2^2)ds}
\bigg[\|\nabla_Hu_0\|_2^2+C\int_0^t(\|u\|_6^2\|\Delta_H\bar v\|_{L^2(M)}^2\\
&+\|\nabla_HT\|_2^2+\|\nabla_H\bar v\|_{L^2(M)}^2\|\Delta_H\bar v\|_{L^2(M)}^2+\|\partial_zu\|_2^2\|\nabla\partial_zu\|_2^2)ds\bigg]\\
\leq&Ce^{C(K_1^2(t)t+K_2'^2(t))}(\|v_0\|_{H^2}^2+K_1^2(t)+K_1(t)+K_2'^2(t))=:K_2''(t).
\end{align*}

Combining this inequality with (\ref{3.2}), we have
\begin{equation*}
\sup_{0\leq s\leq t}\|\nabla u\|_2^2+\int_0^t\|\nabla^2u\|_2^2ds\leq K_2'(t)+K_2''(t)=:K_2(t),
\end{equation*}
for any $t\in[0,\mathcal T^*)$, completing the proof.
\end{proof}

\begin{proposition}\label{lem3.3}
There is a bounded continuously increasing function $K_3(t)$, on $[0,\mathcal T^*)$, such that
\begin{align*}
\sup_{0\leq s\leq t}(\|\Delta_Hv\|_2^2+\|\Delta T\|_2^2)+\int_0^t(\|\nabla\Delta_Hv\|_2^2+\|\nabla_H\Delta T\|_2^2)\leq K_3(t),
\end{align*}
for any $t\in[0,\mathcal T^*)$.
\end{proposition}

\begin{proof}
It follows from integrating by parts, the Sobolev embedding inequality and the Poincar\'e inequality that
\begin{align*}
&\int_\Omega|\nabla_Hv|^4dxdydz=-\int_\Omega\nabla_H\cdot(|\nabla_Hv|^2\nabla_Hv)vdxdydz\\
\leq& C\int_\Omega|\nabla_Hv|^2|v||\nabla^2_Hv|dxdydz\leq C\|\nabla_Hv\|_4^2\|v\|_6\|\nabla_H^2v\|_3
\end{align*}
and
\begin{align*}
&\int_\Omega|\nabla^2_Hv|^3dxdydz=-\int_\Omega\nabla_H\cdot(|\nabla_H^2v|\nabla_H^2v)\nabla_Hvdx
dydz\\
\leq&C\int_\Omega|\nabla_H^2v||\nabla_H^3v||\nabla_Hv|dxdydz\leq C\|\nabla_H^2v\|_3\|\nabla_H^3v\|_2\|\nabla_Hv\|_6\\
\leq&C\|\nabla_H^2v\|_3\|\nabla_H^3v\|_2\|\nabla_H\nabla v\|_2=C\|\nabla_H^2v\|_3\|\nabla_H\Delta_Hv\|_2\|\nabla_H\nabla v\|_2.
\end{align*}
The above two inequalities imply
$$
\|\nabla_Hv\|_4^2\leq C\|v\|_6\|\nabla_H^2v\|_3,\quad \|\nabla_Hv\|_3\leq C\|\nabla_H\Delta_Hv\|_2^{1/2}\|\nabla_H\nabla v\|_2^{1/2},
$$
and thus
\begin{equation}\label{3.4}
\|\nabla_Hv\|_4^2\leq C\|v\|_6\|\nabla_H\nabla v\|_2^{1/2}\|\nabla_H\Delta_Hv\|_2^{1/2}.
\end{equation}

Applying the operator $\nabla_H$ to equation (\ref{1.8}), multiplying the resulting
equation by $-\nabla_H\Delta_Hv$ and integrating over $\Omega$, then it follows from Lemma \ref{lem2.0}, (\ref{1.9}), (\ref{3.4}), the H\"older, Sobolev and Poincar\'e inequalities that
\begin{align*}
&\frac{1}{2}\frac{d}{dt}\int_\Omega|\Delta_Hv|^2dxdydz+\int_\Omega|\nabla\Delta_Hv|^2dxdydz
\nonumber\\
=&\int_\Omega\nabla_H\bigg[(v\cdot\nabla_H)v-\left(\int_{-h}^z\nabla_H\cdot vd\xi\right)\partial_zv-\nabla_H\left(\int_{-h}^zTd\xi\right)\bigg]:\nabla_H\Delta_Hvdxdydz\nonumber\\
\leq&C\int_\Omega\bigg[|v||\nabla_H^2v|+|\nabla_Hv|^2+\left(\int_{-h}^h|\nabla_H^2v|d\xi\right)
|\partial_zv|+\left(\int_{-h}^h|\nabla_H\cdot v|d\xi\right)|\partial_z\nabla_Hv|\nonumber\\
&+\left(\int_{-h}^h|\nabla_H^2T|d\xi\right)\bigg]|\nabla_H\Delta_Hv|dxdydz\nonumber\\
\leq&C(\|v\|_6\|\nabla_H^2v\|_3+\|\nabla_Hv\|_4^2+\|\nabla_H^2v\|_3\|\partial_zv\|_6
+\|\nabla_Hv\|_3
\|\nabla_Hu\|_6\nonumber\\
&+\|\nabla_H^2T\|_2)\|\nabla_H\Delta_Hv\|_2\nonumber\\
\leq&C\big[(\|u\|_6+\|v\|_6)\|\nabla_H^2v\|_2^{1/2}(\|\nabla_H^2v\|_2^{1/2}+\|\nabla\nabla_H^2v
\|_2^{1/2})\nonumber\\
&+\|v\|_6\|\nabla_H\nabla v\|_2^{1/2}\|\nabla_H\Delta_Hv\|_2^{1/2}+\|\nabla\nabla_Hv\|_2\|\nabla
\nabla_Hu\|_2+\|\Delta_H T\|_2\big]\|\nabla_H\Delta_Hv\|_2\nonumber\\
\leq&C\big[(\|v\|_6+\|u\|_6)\|\Delta_Hv\|_2^{1/2}\|\nabla\Delta_Hv\|_2^{1/2}+\|v\|_6(\|\Delta_Hv
\|_2^{1/2}+\|\nabla_Hu\|_2^{1/2})\nonumber\\
&\times\|\nabla_H\Delta_Hv\|_2^{1/2}+(\|\Delta_Hv\|_2+\|\nabla_Hu\|_2)\|\nabla^2u\|_2+\|\Delta_HT\|_2\big]\|\nabla_H\Delta_Hv\|_2\nonumber\\
\leq&\frac{1}{2}\|\nabla\Delta_Hv\|_2^2+C[(\|u\|_6^4+\|v\|_6^4+\|\nabla^2u\|_2^2)(\|\Delta_Hv\|_2^2
+\|\nabla_Hu\|_2^2)+\|\Delta_HT\|_2^2],
\end{align*}
and thus
\begin{align}
  &\frac{d}{dt}\|\Delta_Hv\|_2^2+\|\nabla\Delta_Hv\|_2^2\nonumber\\
  \leq&C[(\|u\|_6^4+\|v\|_6^4+\|\nabla^2u\|_2^2)(\|\Delta_Hv\|_2^2
+\|\nabla_Hu\|_2^2)+\|\Delta_HT\|_2^2]. \label{3.5}
\end{align}

Applying the operator $\nabla$ to equation (\ref{1.10}), multiplying the resulting equation by $-\nabla\Delta T$ and integrating over $\Omega$, using the facts that
\begin{eqnarray*}
&&|\Delta v(x,y,z,t)|\leq|\Delta_H\bar v(x,y,t)|+\int_{-h}^h|\Delta u(x,y,\xi,t)|d\xi,\\
&&|\nabla v(x,y,z,t)|\leq|\nabla_H\bar v(x,y,t)|+\int_{-h}^h|\nabla u(x,y,\xi,t)|d\xi,\\
&&|\nabla\nabla_H\cdot v(x,y,z,t)|\leq\int_{-h}^h|\nabla\nabla_H\cdot u(x,y,\xi,t)|d\xi,
\end{eqnarray*}
then it follows from integrating by parts, Lemma \ref{lem2.2}, Lemma \ref{lem2.0}, the H\"older, Sobolev and
Poincar\'e inequalities that
\begin{align}
&\frac{1}{2}\frac{d}{dt}\int_\Omega|\Delta T|^2dxdydz+\int_\Omega|\nabla_H\Delta T|^2dxdydz\nonumber\\
=&-\int_\Omega\Delta\bigg[v\cdot\nabla_HT-\left(\int_{-h}^z\nabla_H\cdot vd\xi\right)\left(\partial_zT+\frac{1}{h}\right)\bigg]\Delta Tdxdydz\nonumber\\
=&-\int_\Omega\bigg[\Delta v\cdot\nabla_HT+2\nabla v:\nabla_H\nabla T-\left(\int_{-h}^z\Delta\nabla_H\cdot vd\xi\right)\left(\partial_zT+\frac{1}{h}\right)\nonumber\\
&-2\left(\int_{-h}^z\nabla\nabla_H\cdot vd\xi\right)\cdot\nabla\partial_zT\bigg]\Delta Tdxdydz\nonumber\\
=&-\int_\Omega\bigg[\Delta v\cdot\nabla_HT+2\nabla v:\nabla_H\nabla T-\frac{1}{h}\left(\int_{-h}^z\Delta\nabla_H\cdot vd\xi\right)\nonumber\\
&-2\left(\int_{-h}^z\nabla\nabla_H\cdot vd\xi\right)\nabla\partial_zT\bigg]\Delta Tdxdydz\nonumber\\
&+\int_\Omega\left(\int_{-h}^z\Delta\nabla_H\cdot vd\xi\right)\partial_zT(\Delta_HT+\partial_z^2T)dxdydz\nonumber\\
=&-\int_\Omega\bigg[\Delta v\cdot\nabla_HT+2\nabla v:\nabla_H\nabla T-\frac{1}{h}\left(\int_{-h}^z\Delta\nabla_H\cdot vd\xi\right)\nonumber\\
&-2\left(\int_{-h}^z\nabla\nabla_H\cdot vd\xi\right)\cdot\nabla\partial_zT\bigg]\Delta Tdxdydz-\int_\Omega\bigg[\Delta\nabla_H\cdot vT\Delta_HT\nonumber\\
&+\left(\int_{-h}^z\Delta\nabla_H\cdot vd\xi\right)T\Delta_H\partial_zT+\frac{1}{2}\Delta\nabla_H\cdot v|\partial_zT|^2\bigg]dxdydz\nonumber\\
=&-\int_\Omega\bigg[\Delta v\cdot\nabla_HT+2\nabla v:\nabla_H\nabla T-\frac{1}{h}\left(\int_{-h}^z\Delta\nabla_H\cdot vd\xi\right)\nonumber\\
&-2\left(\int_{-h}^z\nabla\nabla_H\cdot vd\xi\right)\nabla\partial_zT\bigg]\Delta Tdxdydz-\int_\Omega\bigg[\Delta\nabla_H\cdot vT\Delta_HT\nonumber\\
&+\left(\int_{-h}^z\Delta\nabla_H\cdot vd\xi\right)T\Delta_H\partial_zT-\Delta v\cdot\nabla_H\partial_zT\partial_zT\bigg]dxdydz\nonumber\\
\leq&C\int_\Omega\bigg[\left(\int_{-h}^h|\Delta u|d\xi+|\Delta_H\bar v|\right)|\nabla_HT|+\left(\int_{-h}^h|\nabla u|d\xi+|\nabla_H\bar v|\right)|\nabla_H\nabla T|\nonumber\\
&+\left(\int_{-h}^h|\nabla_H\Delta v|d\xi\right)+\left(\int_{-h}^h|\nabla\nabla_H\cdot u|d\xi\right)|\nabla\partial_zT|\bigg]|\Delta T|dxdydz\nonumber\\
&+C\int_\Omega\bigg[|\Delta\nabla_Hv||T||\Delta_HT|+\left(\int_{-h}^h|\Delta\nabla_Hv|d\xi\right)
|T||\Delta_H\partial_zT|\bigg]dxdydz\nonumber\\
&+C\int_\Omega\left(\int_{-h}^h|\Delta u|d\xi+|\Delta_H\bar v|\right)|\partial_zT||\nabla_H\partial_zT|dxdydz\nonumber\\
\leq&C\int_M\bigg[\int_{-h}^h(|\nabla u|+|\nabla^2u|)d\xi\bigg]\bigg[\int_{-h}^h(|\nabla T|+|\nabla^2T|)|\nabla^2 T|d\xi\bigg]dxdy\nonumber\\
&+C\int_M(|\nabla_H\bar v|+|\Delta_H\bar v|)\left(\int_{-h}^h(|\nabla T|+|\nabla^2T|)|\nabla^2 T|d\xi\right)dxdy\nonumber\\
&+C\int_M\left(\int_{-h}^h|\nabla_H\Delta v|d\xi\right)|\Delta T|dxdydz+C\int_\Omega\bigg[|\Delta\nabla_Hv||T||\Delta_HT|\nonumber\\
&+\left(\int_{-h}^h|\Delta\nabla_Hv|d\xi\right)|T||\Delta_H\partial_zT|\bigg]dxdydz\nonumber\\
\leq&C[(\|\nabla u\|_2+\|\nabla^2u\|_2+\|\nabla_H\bar v\|_{L^2(M)}+\|\Delta_H\bar v\|_{L^2(M)})(\|\nabla T\|_2^{1/2}+\|\nabla^2T\|_2^{1/2})
\nonumber\\
&\times(\|\nabla T\|_2^{1/2}+\|\nabla^2T\|_2^{1/2}+\|\nabla_H\nabla T\|_2^{1/2}+\|\nabla_H\nabla^2T\|_2^{1/2})\nonumber\\
&\times\|\nabla^2 T\|_2^{1/2}(\|\nabla^2 T\|_2^{1/2}+\|\nabla_H\nabla^2 T\|_2^{1/2})+\|\nabla_H\Delta v\|_2\|\Delta T\|_2\nonumber\\
&+\|\nabla_H\Delta v\|_2\|T\|_\infty\|\Delta_HT\|_2+\|\nabla_H\Delta v\|_2\|T\|_\infty\|\Delta_H\partial_zT\|_2]\nonumber\\
\leq&C[(\|\nabla^2u\|_2+\|\Delta_H\bar v\|_{L^2(M)})\|\Delta T\|_2(\|\Delta T\|_2+\|\nabla_H\Delta T\|_2)+\|\nabla_H\Delta v\|_2\|\Delta T\|_2\nonumber\\
&+\|\Delta\nabla_H v\|_2\|T\|_\infty\|\Delta_HT\|_2+\|\nabla_H\Delta v\|_2\|T\|_\infty\|\nabla_H\Delta T\|_2]\nonumber\\
\leq&C(1+\|\nabla^2u\|_2^2+\|\Delta_H\bar v\|_2^2)\|\Delta T\|_2^2+\frac{1}{2}\|\nabla_H\Delta T\|_2^2+C(1+\|T\|_\infty^2)\|\nabla_H\Delta v\|_2^2,\nonumber
\end{align}
and thus
\begin{align}
&\frac{d}{dt}\|\Delta T\|_2^2+\|\nabla_H\Delta T\|_2^2\nonumber\\
\leq &C(1+\|T\|_\infty^2)\|\nabla_H\Delta v\|_2^2+C(1+\|\nabla^2u\|_2^2+\|\Delta_H\bar v\|_2^2)\|\Delta T\|_2^2.\label{3.6}
\end{align}

For any given $t\in[0,\mathcal T^*)$, recalling that $K_1(t)$ is a bounded continuously increasing function, it follows from Proposition \ref{lem3.1} that
$$
\sup_{0\leq s\leq t}\|T\|_\infty^2\leq\sup_{0\leq s\leq t}K_1(s)\leq K_1(t).
$$
Therefore it follows from (\ref{3.6}) that
\begin{align}
&\frac{d}{ds}\|\Delta T(s)\|_2^2+\|\nabla_H\Delta T(s)\|_2^2\nonumber\\
\leq&C(1+\|\nabla^2u(s)\|_2^2+\|\Delta_H\bar v(s)\|_{L^2(M)}^2)\|\Delta T(s)\|_2^2\nonumber\\
&+C(1+K_1(t))\|\nabla_H\Delta v(s)\|_2^2, \label{3.7}
\end{align}
for all $s\in(0,t)$, with $t\in(0,\mathcal T^*)$.
On the other hand, by (\ref{3.5}), it holds that
\begin{align}
&\frac{d}{ds}\|\Delta_Hv(s)\|_2^2+\|\nabla\Delta_Hv(s)\|_2^2
\nonumber\\
\leq&C(1+\|u(s)\|_6^4+\|v(s)\|_6^4+\|\nabla^2u(s)\|_2^2)(\|\Delta_Hv(s)\|_2^2+\|\Delta_HT(s)\|_2^2)
\nonumber\\
&+C(\|u(s)\|_6^4+\|v(s)\|_6^4+\|\nabla^2u(s)\|_2^2)\|\nabla_Hu(s)\|_2^2. \label{3.8}
\end{align}
Choose a sufficiently big positive constant $\alpha$. Multiplying (\ref{3.8}) by $\alpha(1+K_1(t))$ and summing the resulting inequality up with (\ref{3.7}), then we obtain
\begin{align*}
&\frac{d}{ds}[\alpha(1+K_1(t))\|\Delta_Hv(s)\|_2^2+\|\Delta T(s)\|_2^2]+(\|\nabla\Delta_Hv(s)\|_2^2+\|\nabla_H\Delta T(s)\|_2^2)\nonumber\\
\leq&C(1+K_1(t))(1+\|u\|_6^4+\|v\|_6^4+\|\Delta_H\bar v\|_{L^2(M)}^2+\|\nabla^2u\|_2^2)(s)\nonumber\\
&\times(\|\Delta_Hv\|_2^2+\|\Delta T\|_2^2)(s)+C(1+K_1(t))(\|u \|_6^4+\|v \|_6^4+\|\nabla^2u\|_2^2)(s)\|\nabla_Hu(s)\|_2^2
\end{align*}
for any $0\leq s\leq t<\mathcal T^*$. By Proposition \ref{3.1} and Proposition \ref{lem3.2}, it follows from this inequality that
\begin{align*}
&\sup_{0\leq s\leq t}(\|\Delta_Hv\|_2^2+\|\Delta T\|_2^2)+\int_0^t(\|\nabla\Delta_Hv\|_2^2+\|\nabla_H\Delta T\|_2^2)\\
\leq&Ce^{C(1+K_1(t))\int_0^t(1+\|u\|_6^4+\|v\|_6^4+\|\Delta_H\bar v\|_{L^2(M)}^2+\|\nabla^2u\|_2^2)ds}\\
&\times\bigg[\|v_0\|_{H^2}^2+\|T_0\|_{H^2}^2+(1+K_1(t))\int_0^t(\|u \|_6^4+\|v \|_6^4+\|\nabla^2u\|_2^2)\|\nabla_Hu\|_2^2ds\bigg]\\
\leq&Ce^{C(1+K_1(t))(t+K_1^2(t)t+K_1(t)+K_2(t))}[\|v_0\|_{H^2}^2+\|T_0\|_{H^2}^2\\
&+(1+K_1(t))
(K_1^2(t)t+K_2(t))K_2(t)]=:K_3(t),
\end{align*}
for every $t\in[0,\mathcal T^*)$, completing the proof.
\end{proof}

With these a priori estimates in hand, we are now ready to prove the global existence of strong solutions as follows.

\begin{proof}[\textbf{Proof of Theorem \ref{thm1}}]
By Proposition \ref{prop2.1}, there is a unique strong solution $(v,T)$ in $\Omega\times(0,t_0)$. We consider the solution on the maximal interval of existence $(0,\mathcal T^*)$. We need to prove that $\mathcal T^*=\infty$. Recall that we have assumed by contradiction that $\mathcal T^*<\infty$. By Proposition \ref{lem3.1}, Proposition \ref{lem3.2} and Proposition \ref{lem3.3}, we have the following estimate
\begin{align*}
\sup_{0\leq s\leq t}(\|v(s)\|_{H^2}^2+\|T(s)\|_{H^2}^2)+\int_0^t(\|\nabla v\|_{H^2}^2+\|\nabla_HT\|_{H^2}^2)ds\leq CK(t),
\end{align*}
for any $t\in(0,\mathcal T^*)$, where
$$
K(t)=K_1(t)+K_2(t)+K_3(t).
$$
Note that $K(t)$ is a bounded continuously increasing function, on $(0,\mathcal T^*),$ the above inequality implies that
\begin{align*}
\sup_{0\leq t<\mathcal T^*}(\|v(t)\|_{H^2}^2+\|T(t)\|_{H^2}^2)+\int_0^{\mathcal T^*}(\|\nabla v\|_{H^2}^2+\|\nabla_HT\|_{H^2}^2)dt\leq CK(\mathcal T^*),
\end{align*}
and thus, by Proposition \ref{prop2.1}, we can extend such strong solution beyond $\mathcal T^*$, contradicting to the definition of $\mathcal T^*$. This contradiction implies that $\mathcal T^*=\infty$, and thus completes the proof of Theorem \ref{thm1}.
\end{proof}

\section*{Acknowledgments}
{The work of C.C.  work is  supported in part by NSF grant DMS-1109022. The work of E.S.T. is supported in part by the Minerva Stiftung/Foundation, and  by the NSF
grants DMS-1009950, DMS-1109640 and DMS-1109645.}
\par

\end{document}